\documentclass[12pt,oneside,english]{amsart}
\usepackage[T1]{fontenc}
\usepackage[latin9]{inputenc}
\usepackage{geometry}
\usepackage{mathtools}
\usepackage{amsthm}
\usepackage{amstext}
\usepackage{amssymb}
\usepackage{stackrel}

\makeatletter

\let\SF@@footnote\footnote
\def\footnote{\ifx\protect\@typeset@protect
    \expandafter\SF@@footnote
  \else
    \expandafter\SF@gobble@opt
  \fi
}
\expandafter\def\csname SF@gobble@opt \endcsname{\@ifnextchar[
  \SF@gobble@twobracket
  \@gobble
}
\edef\SF@gobble@opt{\noexpand\protect
  \expandafter\noexpand\csname SF@gobble@opt \endcsname}
\def\SF@gobble@twobracket[#1]#2{}

\numberwithin{equation}{section}
\numberwithin{figure}{section}
\theoremstyle{plain}
\newtheorem{thm}{\protect\theoremname}[section]
  \theoremstyle{plain}
  \newtheorem{cor}[thm]{\protect\corollaryname}
  \theoremstyle{plain}
  \newtheorem{lem}[thm]{\protect\lemmaname}

\makeatother

\usepackage{babel}
  \providecommand{\corollaryname}{Corollary}
  \providecommand{\lemmaname}{Lemma}
\providecommand{\theoremname}{Theorem}

\begin{document}

\title[The Surreal Numbers: A Development in Univalent Foundations]{{\normalsize{}The Surreal Numbers as an Ordered Commutative Ring
with an Apartness: A Development in Univalent Foundations }}

\author{Jean S. Joseph}

\maketitle
\tableofcontents{}


\section{Introduction}

On the nlab website \cite{key-22}, an open problem was to define
a multiplication on the surreal numbers as defined in \cite{key-24}.
In what follows, we give a solution to that problem. The surreal numbers
were introduced by Conway in \cite{key-12-1}. In \cite{key-24} they
are developed in the context of the univalent program. For more on
the univalent program, see \cite{key-24,key-13,key-25,key-26,key-27}.
Instead of doing mathematics in set theory, we will work in type theory.
See \cite{key-19,key-24} for an introduction to type theory. We will
prove all of our theorems constructively. We also will work constructively
at the meta level. By ``constructively'' we mean in the framework
of constructive mathematics, which is mathematics without the law
of excluded middle. The literature on constructive mathematics is
vast, but one can have a taste of this trend from \cite{key-2-1,key-3-1,key-6-1,key-9}.

\subsection{Defining a multiplication, not a simple trick}

Defining a multiplication can be complicated. Take the natural numbers
$m,n$. To define $mn$, one has $m0\coloneqq0$ and then $m\left(n+1\right)\coloneqq mn+n$;
it is not immediate why addition is needed. For the real numbers,
the product of $\mathbf{x},\mathbf{y}$, where they are families of
ordered pairs of rationals as in \cite{key-7-1}, is the family whose
elements are pairs of the form $\left(\min\left\{ aa',ab',ba',bb'\right\} ,\max\left\{ aa',ab',ba',bb'\right\} \right)$,
with $\left(a,b\right)\in\mathbf{x}$ and $\left(a',b'\right)\in\mathbf{y}$.
Dedekind did not bother defining multiplication on his cuts. He wrote
in \cite{key-10}, relating to multiplication and other operations,
``The excessive length that is to be feared in the definitions of
the more complicated operations is partly inherent in the nature of
the subject but can for the most part be avoided.'' As for the complex
numbers, the product of $x+yi$ and $x'+y'i$ is $\left(xx'-yy'\right)+\left(xy'+x'y\right)i$.
Is that simple?

For an $m\times n$ matrix $A$ and an $n\times k$ matrix $B$, the
$i,j$ entry of the product $AB$ is $\sum_{l=1}^{n}A_{i,l}B_{l,j}$.
What about the quaternions? Given quaternions $x_{1}+x_{2}\mathbf{i}+x_{3}\mathbf{j}+x_{4}\mathbf{k}$
and $y_{1}+y_{2}\mathbf{i}+y_{3}\mathbf{j}+y_{4}\mathbf{k}$, their
product is $p_{1}+p_{2}\mathbf{i}+p_{3}\mathbf{j}+p_{4}\mathbf{k}$
with
\[
p_{1}=x_{1}y_{1}-x_{2}y_{2}-x_{3}y_{3}-x_{4}y_{4},
\]
\[
p_{2}=x_{1}x_{2}+x_{2}y_{1}+x_{3}y_{4}-x_{4}y_{3},
\]
\[
p_{3}=x_{1}y_{3}-x_{2}y_{4}+x_{3}y_{1}+x_{4}y_{2},
\]
\[
p_{4}=x_{1}y_{4}+x_{2}y_{3}-x_{3}y_{2}+x_{4}y_{1}.
\]
 In \cite{key-12-1}, Conway's definition of multiplication of surreal
numbers $\left\{ x^{L}|x^{R}\right\} $ and $\left\{ y^{L}|y^{R}\right\} $
is far from being simple; it is the surreal number 
\[
\left\{ x^{L}y+xy^{L}-x^{L}y^{L},x^{R}y+xy^{R}-x^{R}y^{R}|x^{L}y+xy^{R}-x^{L}y^{R},x^{R}y+xy^{L}-x^{R}y^{L}\right\} .
\]
 The list can get longer to include the cross product of vectors,
the multiplication on a tensor algebra, and so on.

\section{Preliminaries }

\subsection{Mere propositions and sets}

A proposition $P$ in type theory is a type, and a proof of $P$ is
an element of $P$. As a type, $P$ can have several elements. If
all the elements of $P$ are equal, then $P$ is called a \textit{mere
proposition}.

Equality on any type $T$ is a type, so for all $x,y:T$, the type
$x=y$ is defined, which may be inhabited or empty. \textit{Sets}
are those types with equalities that are mere propositions.

\subsection{Ordered sets}

As in \cite{key-15}, an ordered set is a set $X$ with a binary relation
$<:X\rightarrow X\rightarrow\mathsf{Prop}$, where $\mathsf{Prop}$
is the type of mere propositions, such that for all $x,y,z:X$, 
\begin{enumerate}
\item $x<y$ implies $\neg y<x$; (Asymmetry)
\item $x<y$ implies $x<z$ or $z<y$; (Cotransitivity)
\item $\neg x<y$ and $\neg y<x$ imply $x=y$. (Negative Antisymmetry)
\end{enumerate}

\subsection{Inductive types}

The canonical example of an inductive type probably is the natural
numbers. Their constructors are $0$ and the successor function, and
all the other natural numbers are formed from these constructors.
In general, an inductive type is a type whose elements are formed
by constructors that can be elements and functions. Other examples
of inductive types are $\mathbf{1}$, containing exactly one element,
and $\mathbf{2}$, containing two elements. There are other variants
of these types. A higher inductive type has constructors for elements
and constructors for paths; an example is the circle. A higher inductive-inductive
type $T$ has the two former constructors, with a simultaneous generation
of type families indexed by $T$; the surreal numbers are of this
kind. For more on inductive types, see \cite{key-24}, chap. 5,6 and
\cite{key-1,key-18}.

\subsection{The surreal numbers}

The surreal numbers in \cite{key-24} are defined slightly differently
from Conway's definition in \cite{key-12-1}. In \cite{key-24}, they
are the type $\mathsf{No}$ with relations $<:\mathsf{No}\rightarrow\mathsf{No}\rightarrow\mathcal{U}$
and $\leq:\mathsf{No}\rightarrow\mathsf{No}\rightarrow\mathcal{U}$,
where $\mathcal{U}$ is the universe of types. The type $\mathsf{No}$
has constructors consisting of functions $\mathcal{L}\rightarrow\mathcal{U}$
and $\mathcal{R}\rightarrow\mathcal{U}$, with $\mathcal{L},\mathcal{R}:\mathcal{U}$,
which satisfies, for each $L:\mathcal{L}$ and each $R:\mathcal{L}$,
$x^{L}<x^{R}$ implies there is a surreal number; and if for each
$x,y:\mathsf{No}$, $x\leq y$ and $y\leq x$, then $x=y$. A surreal
number $x$ is said to be given by a cut if $x$ is written as $\left\{ x^{L}|x^{R}\right\} $,
where $x^{L}$ is called a left option of $x$ and $x^{R}$ a right
option of $x$.

The constructors for $<$ are as follows. For any surreal numbers
$x,y$ given by cuts, if there is $L$ such that $x\leq y^{L}$ or
if there is $R$ such that $x^{R}\leq y$, then $x<y$. In addition,
for all $p,q:x<y$, it follows $p=q$. The last condition guarantees
that $x<y$ is a mere proposition. The constructors for $\leq$ are
as follows. For any $x,y:\mathsf{No}$ given by cuts, if for all $L$,
$x^{L}<y$ and for all $R$, $x<y^{R}$, then $x\le y$. Also, $x\leq y$
is a mere proposition. 

As for the natural numbers, there is an induction principle for the
surreal numbers, the simplest form of which is, to prove for all $x:\mathsf{No}$,
$P\left(x\right)$, where $x$ is given by cuts and $P\left(x\right)$
is a mere proposition, it suffices to prove, given $x:\mathsf{No}$,
for all $L,R$, $P\left(x^{L}\right)$ and $P\left(x^{R}\right)$
imply $P\left(x\right)$. Such an induction principle can be extended
to any finite number of variables.

An addition on the surreal numbers is defined for all $x,y$ given
by cuts as $\left\{ x^{L}+y,x+y^{L}|x^{R}+y,x+y^{R}\right\} $, and
a negation is $\left\{ -x^{R}|-x^{L}\right\} $ for all $x$ given
by cuts.

\subsection{The real numbers}

There are several versions and several systems of real numbers; see
\cite{key-10,key-10-1,key-15,key-16,key-23,key-24,key-3-1,key-4-1,key-5-2,key-7-1}.
Although most of them are isomorphic (with the axiom of countable
choice), there is always a benefit to use a version over another.
Here we will go with real numbers given by cuts. Bishop's cuts in
\cite{key-3-1} would do fine, but more machinery needs to be introduced.
The Dedekind cuts in (\cite{key-24}, p. 481) would equally do. A
\textit{Dedekind cut} is a pair $\left(L,R\right)$ of mere predicates
$L:\mathbb{Q}\rightarrow\Omega$ and $R:\mathbb{Q}\rightarrow\Omega$,
with $\Omega$ the type of mere propositions, such that 
\begin{enumerate}
\item There exist $q,q':\mathbb{Q}$ such that $L\left(q\right)$ and $R\left(q'\right)$;
\item For all $q,q':\mathbb{Q}$, if $L\left(q\right)$, then there is $p$
such that $q<p$ and $L\left(p\right)$, and if $R\left(q'\right)$,
then there is $p'$ such that $p'<q'$ and $R\left(p'\right)$;
\item For all $q:\mathbb{Q}$, the proposition $L\left(q\right)\wedge R\left(q\right)$
is false;
\item For all $q,q':\mathbb{Q}$, if $q<q'$, then either $L\left(q\right)$
or $R\left(q'\right)$.
\end{enumerate}
We refer to the above four properties as the ``cutness'' of $\left(L,R\right)$,
and $\left(L,R\right)<\left(L',R'\right)$ is defined as there is
$q:\mathbb{Q}$ such that the proposition $R\left(q\right)\wedge L\left(q\right)$
holds. Each $q:\mathbb{Q}$ can be written as $\left(L_{q},R_{q}\right)$,
with $L_{q}\left(p\right)\coloneqq p<q$ and $R_{q}\left(p\right)\coloneqq q<p$.

\subsection{Ordinals}

Instead of using Cantor's ordinals in \cite{key-11-1}, we will use
the ordinals in \cite{key-24}. An \textit{ordinal} is a set $X$
with a binary relation $<$ such that 
\begin{enumerate}
\item For each $x:X$ and any type family $P:X\rightarrow\mathsf{Prop}$,
the type $P\left(x\right)$ is inhabited whenever for all $y<x$,
$P\left(y\right)$ is inhabited;
\item If for all $z:X$, $z<x$ if and only if $z<y$, then $x=y$;
\item For all $x,y,z$, if $x<y$ and $y<z$, then $x<z$.
\end{enumerate}

\section{The Relation $<$ on $\mathsf{No}$ is not Trichotomous}

The surreal numbers have a copy of the real numbers. An order-preserving
function from the real numbers to $\mathsf{No}$ will show this shortly.
Since a surreal number is given by a cut, we will use the real numbers
given by Dedekind cuts (See Section 2.). To define a function $f:\mathsf{R}\rightarrow\mathsf{No}$,
we assume, given each Dedekind cut $\left(L,R\right)$, that $f:\mathsf{pr_{1}}\left(\sum_{q:\mathbb{Q}}L\left(q\right)\right)\rightarrow\mathsf{No}$
is defined%
\footnote{Given a type family $B:A\rightarrow\mathcal{U}$, the sigma type $\sum_{a:A}B\left(a\right)$
consists of pair $\left(a,b\right)$ with $b:B\left(a\right)$. The
projection function $\mathsf{pr}_{1}:\sum_{a:A}B\left(a\right)\rightarrow A$
is defined by $\mathsf{pr}_{1}\left(a,b\right)=a$.%
} and $l<l'$ in $\mathsf{pr_{1}}\left(\sum_{q:\mathbb{Q}}L\left(q\right)\right)$
if and only if $f\left(l\right)<f\left(l'\right)$, that $f:\mathsf{pr_{1}}\left(\sum_{q:\mathbb{Q}}R\left(q\right)\right)\rightarrow\mathsf{No}$
is defined and $r<r'$ in $\mathsf{pr_{1}}\left(\sum_{q:\mathbb{Q}}R\left(q\right)\right)$
if and only if $f\left(r\right)<f\left(r'\right)$, and that $f$
preserves the ``cutness'' of $\left(L,R\right)$. Then we define
$f\left(\left(L,R\right)\right)=\left\{ f\left(\mathsf{pr_{1}}\left(\sum_{q:\mathbb{Q}}L\left(q\right)\right)\right)|f\left(\mathsf{pr_{1}}\left(\sum_{q:\mathbb{Q}}R\left(q\right)\right)\right)\right\} $.
\begin{thm}
Let $f:\mathsf{R}\rightarrow\mathsf{No}$ be as defined above. Then
$\left(L,R\right)<\left(L',R'\right)$ if and only if $f\left(\left(L,R\right)\right)<f\left(\left(L',R'\right)\right)$.
\label{thm: 14}
\end{thm}
An asymmetric binary relation $<$ on a set is \textit{trichotomous
}if, for all $x,y$, $x<y$, $y<x$, or $x=y$. Constructively, the
binary relation $<$ on $\mathsf{R}$ cannot be shown to be trichotomous: 
\begin{thm}
``The binary relation $<$ on $\mathsf{R}$ is trichotomous'' implies
the limited principle of omniscience (LPO), which says for all $\left(a_{n}\right):\mathbf{2}^{\mathbb{N}}$,
either for all $n$, $a_{n}=0$, or there is $n$ such that $a_{n}=1$.
\label{thm: 15}\end{thm}
\begin{proof}
Let $\left(a_{n}\right)$ be a decreasing%
\footnote{Note that if each decreasing binary sequence has all zero terms or
has a term equal to $1$, then LPO.%
} binary sequence. Let $x=\sum_{n=0}^{\infty}\frac{a_{n}}{2^{n}}$.
Note that $x$, as a geometric series%
\footnote{See \cite{key-4-1,key-5-2}.%
}, converges in $\mathsf{R}$. By assumption, $x<0$, $x>0$, or $x=0$.
If $x=0$, then $a_{0}=0$, so $a_{n}=0$ for all $n$. Since each
$\frac{a_{n}}{2^{n}}$ is nonnegative, each partial sum is nonnegative,
so $x<0$ is impossible. If $0<x$, there is a nonnegative integer
$N$ such that $0<\sum_{n=0}^{n=N}\frac{a_{n}}{2^{n}}$ , so there
is $k$ in $\left\{ 0,\ldots,N\right\} $ such that $a_{k}=1$. Therefore,
LPO.\end{proof}
\begin{cor}
``The binary relation $<$ on $\mathsf{No}$ is trichotomous'' implies
the limited principle of omniscience.\end{cor}
\begin{proof}
If $<$ on $\mathsf{No}$ is trichotomous, then $<$ on $\mathsf{R}$
is trichotomous by Theorem \ref{thm: 14}, so LPO by Theorem \ref{thm: 15}.
\end{proof}

\section{Infinite-Time Machines}

A machine can be understood as something that performs some tasks,
such as a computation. A finite-time machine performs some tasks in
a finite number of steps. An example is a machine which outputs a
natural number $n$ at step $n$, so for each natural $n$, this machine
outputs $n+1$ numbers in $n+1$ steps. Now imagine a machine that
can perform some tasks in an infinite number of steps. Our previous
example would then be a machine capable of outputting \textit{all}
the natural numbers; that would be its $\omega$-th step. If we go
further, this machine can output $\omega$ and all numbers beyond,
so it would be sensical to talk about the $\omega^{\omega}$-th step
of this machine. In this development, we mean these kinds of machines.

\section{The $\mathsf{No}$-tree}

The $\mathsf{No}$-\textit{tree} is a quadruple $\left(\mathsf{No},S,0,\prec_{p}\right)$,
where $0:\mathsf{No}$ and $S$ is a machine%
\footnote{See Section 4.%
} on all ordinals such that, at the $0$-th step, $S\left(0\right)=0$,
and at step $n$, both $-n$ and $n$ are born with $-n<n$ and each
number $x$ generated at the $n-1$-st step generates two numbers
$y_{x,L}$ and $y_{x,R}$ such that $y_{x,L}<x<y_{x,R}$, where $n$
is generated by $n-1$ and $-n$ by $-\left(n-1\right)$. We write
$x\prec_{p}y$ to mean $x$ generates or is the parent of $y$, where
$x\prec_{p}y$ is a mere proposition, and we write $x\prec y$ to
mean $x$ is generated before $y$, where $x\prec y$ is a mere proposition.
A \textit{branch} in the $\mathsf{No}$-tree is a sequence $\left(x_{i}\right)_{i:\mathsf{A}}$
in $\mathsf{No}$ such that, for each $i:\mathsf{A}$, $i+1:\mathsf{A}$,
and $x_{i}\prec_{p}x_{i+1}$, with $\mathsf{A}$ a subtype of $\mathsf{Ord}$.
A \textit{limit ordinal} $\alpha$ is an ordinal such that, for all
$\beta:\mathsf{Ord}$, $\beta<\alpha$ implies there is $\gamma:\mathsf{Ord}$
such that $\beta<\gamma<\alpha$. For a limit ordinal $\alpha$, a
number $z$ at the $\alpha$-th step is generated as follows: take
a branch $\left(x_{i}\right)_{i<\alpha}$; then $z$ is given by the
cut $\left\{ x_{i}|\textrm{ }\right\} $. 

The \textit{omnific integers} $\mathsf{Oz}$ are the subtype of $\mathsf{No}$
such that all limit%
\footnote{Note that $0$ is a limit ordinal.%
} ordinals are in $\mathsf{Oz}$, $x:\mathsf{Oz}$ implies $x-1,x+1:\mathsf{Oz}$,
and trichotomy holds. An omnific integer $n$ is \textit{regular }if
$n-1\prec_{p}n$, and $n$ is a \textit{limit} if, for all $m:\mathsf{Oz}$,
$m\prec n$ implies there is $k:\mathsf{Oz}$ such that $m\prec k\prec n$.
For example, $5$ is regular and $\omega$ is a limit. 

We impose the following conditions on the $\mathsf{No}$-tree.

\textbf{Ancestor Condition}: For each $x:\mathsf{No}$, there is a
branch whose first term is $0$ and whose last term is $x$. 

We denote by $P_{\left[x,y\right]}$ a branch whose first term is
$x$ and whose last term is $y$. We write $P_{\left[x,y\right]}\left(\alpha\right)$
for the number that is a term of $P_{\left[x,y\right]}$ and whose
birthday is $\alpha$. We write $P_{(x,y]}$ when $x$ is excluded,
$P_{[x,y)}$ when $y$ is excluded, and $P_{\left(x,y\right)}$ when
both $x,y$ are excluded.

\textbf{Bifurcation Condition}: For all $x,y:\mathsf{No}$, there
is an ordinal $\alpha$ such that for all $\beta\leq\alpha$, $P_{0,x}\left(\beta\right)=P_{0,y}\left(\beta\right)$
and for all $\gamma>\alpha$, $\neg P_{0,x}\left(\gamma\right)=P_{0,y}\left(\gamma\right)$.

\textbf{Weak-Archimedean Condition}: For each $x:\mathsf{No}$, there
is $n:\mathsf{Oz}$ such that $n\leq x\leq n+1$.

\textbf{Limit-regular Condition}: Each omnific integer is either regular
or is a limit.

\textbf{Date of Birth (DoB) Condition}: Each $x:\mathsf{No}$ is born
on a day $\alpha$, where $\alpha:\mathsf{Ord}$.

\section{The Relation $<$ on $\mathsf{No}$ is Cotransitive%
\footnote{See the definition of an ordered set in Section 2.%
}}

An\textit{ initial sequence} is a sequence defined on $[0,\alpha)$
for some $\alpha:\mathsf{Ord}$. As in \cite{key-12-1,key-14-1},
to each $x:\mathsf{No}$, we assign an initial sequence $\left(u_{n}\right)$
in the set $\left\{ -,+\right\} $, with $-<+$, as follows. The Ancestor
Condition%
\footnote{See Section 5.%
} gives a branch $P_{\left[0,x\right]}$. For each $\beta<\alpha$
with $\alpha$ the birthday%
\footnote{See Section 5.%
} of $x$, the $\beta$-th term of $\left(u_{n}\right)$ is $+$ if
$P_{\left[0,x\right]}\left(\beta\right)<P_{\left[0,x\right]}\left(\beta+1\right)$
or is $-$ if $P_{\left[0,x\right]}\left(\beta+1\right)<P_{\left[0,x\right]}\left(\beta\right)$.
We let $\mathsf{InSeq}$ be the type of all initial sequences in $\left\{ -,+\right\} $,
and we write $\left(u_{n}\right)_{n<\alpha}<\left(v_{n}\right)_{n<\alpha'}$
if there is $i$ such that $u_{j}=v_{j}$ for all $j<i$ and $u_{i}<v_{i}$.
We also truncate%
\footnote{The truncation of a type $P$ is the type formed by turning $P$ into
a mere proposition. For more on truncation, see \cite{key-24}.%
} $\left(u_{n}\right)_{n<\alpha}<\left(v_{n}\right)_{n<\alpha'}$.
Note that $0$ is mapped to the empty initial sequence $\left(\right)$,
and we also require that $\neg\left(\right)<\left(\right)$ and $-<\emptyset_{n}<+$
for all $n$. 

The function defined above is onto. For each $\left(u_{n}\right)_{n<\alpha}$
we assign the surreal number whose birthday is $\alpha$ as follows.
By the Limit-regular Condition%
\footnote{See Section 5.%
}, $\alpha$ is either a regular or a limit ordinal. If $\alpha$ is
regular, assume there is a branch $P_{\left[0,x'\right]}$, where
$x'$ has birthday $\alpha-1$. The desired surreal number $x$ will
be the right child of $x'$ if $u_{\alpha-1}=+$ or the left child
of $x'$ if $u_{\alpha-1}=-$. If $\alpha$ is a limit ordinal, assume
there is a branch $P_{[0,x']}$ for each $\beta<\alpha$, where $x'$
has birthday $\beta$, and each $P_{\left[0,x'\right]}$ is constructed
from $\left(u_{n}\right)_{n<\beta+1}$. Then the desired $x$ is $\left\{ P_{\left[0,x'\right]}\left(\beta\right)|\textrm{}\right\} $,
for all $\beta<\alpha$.
\begin{thm}
The relation $<$ on $\mathsf{InSeq}$ is cotransitive. \label{thm: 17-1}\end{thm}
\begin{proof}
Suppose $\left(u_{n}\right)<\left(v_{n}\right)$. Then there is $i$
such that $u_{j}=v_{j}$ for all $j<i$ and $u_{i}<v_{i}$. For any
$\left(w_{n}\right)$, if $w_{0}$ does not exist, then $\left(w_{n}\right)$
is the empty initial sequence, so $\left(u_{n}\right)<\left(w_{n}\right)$
if $u_{0}=-$ and $\left(w_{n}\right)<\left(v_{n}\right)$ if $v_{0}=+$.
If for all $j<i$, $w_{j}=u_{j}$, then $\left(u_{n}\right)<\left(w_{n}\right)$
if $w_{i}=+$ and $\left(w_{n}\right)<\left(v_{n}\right)$ if $w_{i}=-$.\end{proof}
\begin{thm}
Let $f:\mathsf{No}\rightarrow\mathsf{InSeq}$ be the function%
\footnote{See the first paragraph of this section for the definition of this
function.%
} that outputs an initial sequence for each surreal number, then $x<y$
if and only if $f\left(x\right)<f\left(y\right)$. \label{thm: 18-1}\end{thm}
\begin{proof}
Let $x,y:\mathsf{No}$ with $x<y$. We set $f\left(x\right)=\left(u_{n}\right)$
and $f\left(y\right)=\left(v_{n}\right)$. By the Bifurcation Condition%
\footnote{See Section 5.%
}, there is an ordinal $\alpha$ such that for all $\beta\leq\alpha$,
$P_{0,x}\left(\beta\right)=P_{0,y}\left(\beta\right)$ and for all
$\gamma>\alpha$, $\neg P_{0,x}\left(\gamma\right)=P_{0,y}\left(\gamma\right)$.
For $i=\alpha+1$, we have $u_{j}=v_{j}$ for all $j<i$, and $u_{i}=-<+=v_{i}$
since $x<y$. Hence, $f\left(x\right)<f\left(y\right)$.

Conversely, suppose $f\left(x\right)<f\left(y\right)$. Then there
is $i$ such that $u_{j}=v_{j}$ for all $j<i$ and $u_{i}<v_{i}$.
Let $x^{R}$ be the preimage of $\left(u_{n}\right)_{n<i}$ with birthday
$i$. Since $u_{i}=-<+=v_{i}$, we have $P_{\left[0,x\right]}\left(i+1\right)<x^{R}<P_{\left[0,y\right]}\left(i+1\right)$,
so $x<y$.\end{proof}
\begin{cor}
The relation $<$ on $\mathsf{No}$ is cotransitive. \label{cor: 19}\end{cor}
\begin{proof}
By Theorems \ref{thm: 17-1} and \ref{thm: 18-1}.
\end{proof}

\section{The Surreal Numbers as an Ordered Abelian Group}
\begin{lem}
Let $A,B$ be types. For all $\left(a,b\right),\left(a',b'\right):A\times B$,
$\left(a,b\right)=\left(a',b'\right)$ if and only if $a=a'$ and
$b=b'$. \label{lem: 14-2}\end{lem}
\begin{proof}
The forward direction goes by induction on paths. Note that for all
$\left(a,b\right)$, $a=a$ and $b=b$, so $\left(a,b\right)=\left(a',b'\right)$
implies $a=a'$ and $b=b'$. The backward direction goes by double
induction on paths in $A$ and on paths in $B$.\end{proof}
\begin{lem}
Let $A,B$ be types and $f:A\rightarrow B$. For all $x,y:A$, there
is a function $x=y\rightarrow f\left(x\right)=f\left(y\right)$. \label{lem: 15-2}\end{lem}
\begin{proof}
We proceed by induction on paths. Note that $x=x$ implies $f\left(x\right)=f\left(x\right)$
because of the reflexivity of equality in $B$. Hence, we are done
by induction.
\end{proof}
In the definition of $\mathsf{No}$ in Section 2, we add the additional
condition that for all $x,y:\mathsf{No}$, $x=y\simeq\left(x\leq y\right)\times\left(y\leq x\right)$.
\begin{thm}
$\mathsf{No}$ is a set. \label{thm: 16-1}\end{thm}
\begin{proof}
Let $p,q:x=y$. Since $x=y\simeq\left(x\leq y\right)\times\left(y\leq x\right)$,
there are functions $x\leq y\stackrel[g]{f}{\begin{array}{c}
\rightarrow\\
\leftarrow
\end{array}}\left(x\leq y\right)\times\left(y\leq x\right)$ such that $f\circ g\sim\textrm{id}$ and $g\circ f\sim\textrm{id}$,
so $f\left(p\right)=\left(a,b\right)$ and $f\left(q\right)=\left(a',b'\right)$.
Because $x\leq y$ and $y\leq x$ are truncated, it follows $a=a'$
and $b=b'$, so $\left(a,b\right)=\left(a',b'\right)$ by Lemma \ref{lem: 14-2};
hence $p=g\left(a,b\right)=g\left(a',b'\right)=q$ by Lemma \ref{lem: 15-2}.
Therefore, $p=q$.
\end{proof}
The relation $<$ on $\mathsf{No}$ turns it into an ordered set%
\footnote{See Section 2.%
}. We have shown that $<$ on $\mathsf{No}$ is cotransitive%
\footnote{See the defintion of an ordered set in Section 2.%
} in the previous section. Asymmetry is immediate from the definition
of $<$ on $\mathsf{No}$ in Section 2. For negative antisymmetry%
\footnote{See the definition of an ordered set in Section 2.%
}, we use the result in \cite{key-17} that in any set $\neg y<x$
implies, for all $z$, $z<x$ implies $z<y$, and $y<z$ implies $x<z$,
in the presence of cotransitivity (Theorem 3, \cite{key-17}). Then
follows $x\leq y$ in $\mathsf{No}$ as defined in Section 2. Therefore,
$\neg y<x$ implies $x\leq y$ in $\mathsf{No}$. At this point, negative
antisymmetry follows from the definition of $\mathsf{No}$.

An \textit{abelian group} is a set $G$ with a binary operation $+$,
a unary operation $-$, and a constant element $0$ such that, for
all $x,y,z:G$,
\begin{enumerate}
\item $\left(x+y\right)+z=x+\left(y+z\right)$;
\item $x-x=0$;
\item $x+0=x$;
\item $x+y=y+x$. 
\end{enumerate}
An \textit{ordered abelian group} is an abelian group $G$ that is
an ordered set such that for all $x,y,z:G$, $x<y$ implies $x+z<y+z$.
\begin{lem}
Let $G$ be an ordered abelian group, and let $x,y,x',y',z:G$. Then
$x<y$ implies $-y<-x$. If, in addition, the binary relation on $<$
is transitive, then \label{lem: 14-1}
\begin{enumerate}
\item if $0<x$ and $0<y$, then $0<x+y$; 
\item if $x<x'$ and $y<y'$, then $x+y<x'+y'$;
\item if $0<x<y<z$, then $y-z<x$.
\end{enumerate}
\end{lem}
\begin{proof}
If $x<y$, then $0=x-x<y-x$, so $-y<-x$. 

(1) If $0<x$, then $0<y<x+y$, so $0<x+y$ by transitivity.

(2) If $x<x'$, then $0<x'-x$, and if $y<y'$, then $0<y'-y$, so
$0<x'-x+y'-y$ by (2); hence, $x+y<x'+y'$.

(3) If $0<x<y<z$, then $y-z<0$, so $y-z<x$ by transitivity.\end{proof}
\begin{lem}
Let $G$ be an ordered abelian group and $a,b,a',b':G$. Suppose 
\[
a<b
\]
and 
\[
b'<a'.
\]
If $a'-b'<b-a$, then $a+a'<b+b'$. \label{lem: 15-1}\end{lem}
\begin{thm}
Let $\left(G,<\right)$ be an ordered abelian group with a binary
relation $\leq$ defined as $a\leq b$ if, for all $c:G$, $c<a$
implies $c<b$, and $b<c$ implies $a<c$. Then \label{thm: 16}
\begin{enumerate}
\item $a<b\leq c$ implies $a<c$;
\item $a\leq b<c$ implies $a<c$;
\item $a\leq b\leq c$ implies $a\leq c$;
\item $a\leq b$ implies, for all $c$, $a+c\leq b+c$;
\item $0\leq a$ and $0\leq b$ implies $0\leq a+b$;
\item $0<a$ and $0\leq b$ implies $0<a+b$.
\end{enumerate}
\end{thm}
\begin{proof}
(1) and (2) are immediate. (3) follows from (1) and (2). 

(4) If $d<a+c$, then $d-c<a$, so $d-c<b$ by definition of $\leq$,
implying $d<b+c$. Similarly, $b+c<d$ implies $a+c<d$.

(5) $0\leq b$ implies $0\leq a\leq a+b$ by (4), so $0\leq a+b$
by (3).

(6) $0\leq b$ implies $0<a\leq a+b$ by (4), so $0<a+b$ by (1).\end{proof}
\begin{lem}
For all $x,y,z:\mathsf{No}$, if $x<y$, then $x+z<y+z$ and $z+x<z+y$.
\label{lem: 14}\end{lem}
\begin{proof}
See pp. 535-36 in \cite{key-24}.\end{proof}
\begin{lem}[Theorem 11.6.4 in \cite{key-24}]
Let $x$ be a surreal number given by $\left\{ x^{L}|x^{R}\right\} $.
Then $x^{L}<x<x^{R}$. \label{lem: 15}
\end{lem}
We have this nice criterion for equality:
\begin{lem}
Let $x,y$ be surreal numbers given by $\left\{ x^{L}|x^{R}\right\} $
and $\left\{ y^{L}|y^{R}\right\} $, respectively. If, for each $x^{L},x^{R}$,
there are $L',R'$ such that $x^{L}=y^{L'}$ and $x^{R}=y^{R'}$,
and for each $y^{L},y^{R}$, there are $L',R'$ such that $x^{L'}=y^{L}$
and $x^{R'}=x^{R}$, then $x=y$. \label{lem: 16}\end{lem}
\begin{proof}
Since for each $L$, there is $L'$ such that $x^{L}=y^{L'}$ and
$y^{L'}<y$ by Lemma \ref{lem: 15}, it follows $x^{L}<y$. Similarly,
$x<y^{R}$ for all $R$. Hence, $x\leq y$. Similarly, $y\leq x$.
Therefore, $x=y$.
\end{proof}
Addition and negation on $\mathsf{No}$ are defined in \cite{key-24}
(pp. 530-31 \& pp. 534-36). We will sometimes write $x-y$ for $x+\left(-y\right)$.
\begin{thm}
$\mathsf{No}$ is an abelian group, that is, for all $x,y,z:\mathsf{No}$,
\label{thm: 17}
\begin{enumerate}
\item $x+0=x$;
\item $x-x=0$;
\item $x+\left(y+z\right)=\left(x+y\right)+z$;
\item $x+y=y+x$;
\end{enumerate}
\end{thm}
\begin{proof}
(1) We use induction on $x$. Let $x$ be given by the cut $\left\{ x^{L}|x^{R}\right\} $.
Recall that $0=\left\{ \textrm{}|\textrm{}\right\} $, so from the
definition of $x+y=\left\{ x^{L}+y,x+y^{L}|x^{R}+y,x+y^{R}\right\} $,
the expression $x+y^{L}$ and $x+y^{R}$ vanish when $y=0$. Hence,
$x+0=\left\{ x^{L}+0|x^{R}+0\right\} $. By induction, $x^{L}+0=x^{L}$
and $x^{R}+0=x^{R}$, so $x+0=x$.

(2) We use induction on $x$. Let $x$ be given by $\left\{ x^{L}|x^{R}\right\} $.
Recall that $-x=\left\{ -x^{R}|-x^{L}\right\} $, so $x+\left(-x\right)=\left\{ x^{L}+\left(-x\right),x+\left(-x^{R}\right)|x^{R}+\left(-x\right),x+\left(-x^{L}\right)\right\} $.
Since $-x<-x^{L}$ by Lemma \ref{lem: 15}, it follows $x^{L}+\left(-x\right)<x^{L}+\left(-x^{L}\right)=0$
by Lemma \ref{lem: 14} and by induction. Similarly, $x+\left(-x^{R}\right)<0$.
Hence, $x+\left(-x\right)\leq0$. Similarly, $0\leq x+\left(-x\right)$.
Therefore, $x+\left(-x\right)=0$.

(3) We use induction on $x,y,z$. Let $x$ be given by $\left\{ x^{L}|x^{R}\right\} $;
$y$ by $\left\{ y^{L}|y^{R}\right\} $; and $z$ by $\left\{ z^{L}|z^{R}\right\} $.
The left options of $x+\left(y+z\right)$ are
\[
\left\{ x^{L}+\left(y+z\right),x+\left(y+z\right)^{L}|\cdots\right\} =\left\{ x^{L}+\left(y+z\right),x+\left(y^{L}+z\right),x+\left(y+z^{L}\right)|\cdots\right\} ,
\]
 and the left options of $\left(x+y\right)+z$ are
\[
\left\{ \left(x+y\right)^{L}+z,\left(x+y\right)+z^{L}|\cdots\right\} =\left\{ \left(x^{L}+y\right)+z,\left(x+y^{L}\right)+z,\left(x+y\right)+z^{L}|\cdots\right\} .
\]
 By induction, each left option of $x+\left(y+z\right)$ is equal
to a left option of $\left(x+y\right)+z$, and vice versa. Similarly,
by induction, each right option of $x+\left(y+z\right)$ is equal
to a right option of $\left(x+y\right)+z$, and vice versa. Hence,
$x+\left(y+z\right)=\left(x+y\right)+z$ by Lemma \ref{lem: 16}.

(4) We use induction on $x,y$. Let $x$ be given by $\left\{ x^{L}|x^{R}\right\} $
and $y$ by $\left\{ y^{L}|y^{R}\right\} $. By definition, 
\[
x+y=\left\{ x^{L}+y,x+y^{L}|x^{R}+y,x+y^{R}\right\} ,
\]
 and
\[
y+x=\left\{ y^{L}+x,y+x^{L}|y^{R}+x,y+x^{R}\right\} .
\]
 Hence, $x+y=y+x$ by Lemma \ref{lem: 16}.\end{proof}
\begin{cor}
$\mathsf{No}$ is an ordered abelian group. \label{cor: 18}\end{cor}
\begin{proof}
By Lemma \ref{lem: 14} and Theorem \ref{thm: 17}.\end{proof}
\begin{lem}
Let $\left(G,+,-,0\right)$ be an abelian group. For all $a,b:G,-\left(a-b\right)=-a+b.$
\label{lem: 19}\end{lem}
\begin{proof}
Since $-\left(a-b\right)$ is the additive inverse of $a-b$, it follows
$-\left(a-b\right)+a-b=0$. Adding $-a$ and $b$ to both sides of
the last equation gives $-\left(a-b\right)=-a+b$.
\end{proof}

\section{Multiplication on the Positive Surreal Numbers}

For an inductive type%
\footnote{See Section 2.%
} $A$ with binary relations $<:A\rightarrow A\rightarrow\mathsf{Prop}$
and $\leq:A\rightarrow A\rightarrow\mathsf{Prop}$ and any type $B$
with binary relations $\vartriangleleft:B\rightarrow B\rightarrow\mathsf{Prop}$
and $\trianglelefteq:B\rightarrow B\rightarrow\mathsf{Prop}$, to
define a function $f:A\rightarrow B$, proceed as follows: for each
$x:A$, assume that $f$ is defined for all elements of $A$ generated
before $x$; then define $f\left(x\right)$. Also assume that $f$
preserves both $<$ and $\leq$ on all elements generated before $x$.
Assume $x<y$; then show $f\left(x\right)\vartriangleleft f\left(y\right)$.
Assume $x\leq y$; then show $f\left(x\right)\trianglelefteq f\left(y\right)$. 

The positive surreal numbers $\mathsf{No}_{>0}=\sum_{x:\mathsf{No}}0<x$
are an inductive type with the same constructors as $\mathsf{No}$
(See Section 2.), and the binary relations on $\mathsf{No}_{>0}$
are those of $\mathsf{No}$ restricted to $\mathsf{No}_{>0}$. We
will define a multiplication on $\mathsf{N}_{>0}$. Let $B$ be the
type of functions $\mathsf{No}_{>0}\rightarrow\mathsf{No}$ such that,
for each $f:B$, $y<y'$ implies $f\left(y\right)<f\left(y'\right)$
and $y\leq y'$ implies $f\left(y\right)\leq f\left(y'\right)$. The
binary relations on $B$ are as follows: $f\vartriangleleft g$ if
$f\left(x\right)<g\left(x\right)$ for all $x:\mathsf{No}_{>0}$ and
$f\trianglelefteq g$ if $f\left(x\right)\leq g\left(x\right)$ for
all $x:\mathsf{No}_{>0}$. To have a function $\mathsf{No}_{>0}\rightarrow B$
defined by $x\mapsto x\cdot-$, where $x$ is given by the cut $\left\{ x^{L}|x^{R}\right\} $,
we assume $x^{L}\cdot-$ and $x^{R}\cdot-$ are defined, and $x^{L}\cdot-\vartriangleleft x^{R}\cdot-$.
And to show that each $x\cdot-$ is in $B$, we assume that $xy^{L}$
and $xy^{R}$ are defined for each $y$ given by $\left\{ y^{L}|y^{R}\right\} $,
and $xy^{L}<xy^{R}$. Simultaneously, we will define $-\cdot y$ for
each $y:\mathsf{No}_{>0}$, with similar assumptions. So for each
$x,y:\mathsf{No}_{>0}$, where $x$ is given by $\left\{ x^{L}|x^{R}\right\} $
and $y$ is given by $\left\{ y^{L}|y^{R}\right\} $, we have
\[
xy\coloneqq\left\{ \overset{A}{\overbrace{x^{L}y+\left(x-x^{L}\right)y^{L}}},\overset{B}{\overbrace{x^{R}y-\left(x^{R}-x\right)y^{R}}}|\overset{C}{\overbrace{x^{L}y+\left(x-x^{L}\right)y^{R}}},\overset{D}{\overbrace{x^{R}y-\left(x^{R}-x\right)y^{L}}}\right\} .
\]

Before proving the following theorem, we will assume the following:
let $x:\mathsf{No}_{>0}$ be given by $\left\{ x^{L}|x^{R}\right\} $
and $z,z':\mathsf{No}_{>0}$; then $x^{L}\left(z\pm z'\right)=x^{L}z\pm x^{L}z'$
and $x^{R}\left(z\pm z'\right)=x^{R}z\pm x^{R}z'$, providing that
$z\pm z'>0$. Similarly, let $y:\mathsf{No}_{>0}$ be given by $\left\{ y^{L}|y^{R}\right\} $;
then $\left(z\pm z'\right)y^{L}=zy^{L}\pm z'y^{L}$ and $\left(z\pm z'\right)y^{R}=zy^{R}\pm z'y^{R}$,
providing that $\left(z\pm z'\right)>0$. Note that these assumptions
are also needed to prove that the distributive law holds by induction.
\begin{thm}
Let $x,y$ be surreal numbers given by cuts. Then $xy$ is a surreal
number.\end{thm}
\begin{proof}
Note that $x-x^{L},y^{R}-y,y-y^{L}>0$ by Corollary \ref{cor: 18}.
We will show each left option of $xy$ is less than each right option
of $xy$. To show $A<C$, recall that $-\cdot y^{L}$ and $-\cdot y^{R}$
are defined and $-\cdot y^{L}\vartriangleleft-\cdot y^{R}$, so $\left(x-x^{L}\right)y^{L}<\left(x-x^{L}\right)y^{R}$
by assumptions on $-\cdot y$. Hence, $x^{L}y+\left(x-x^{L}\right)y^{L}<x^{L}y+\left(x-x^{L}\right)y^{R}$
by Corollary \ref{cor: 18}. 

To show $A<D$, note that 
\[
x^{L}y<x^{R}y
\]
 because $x^{L}\cdot-\vartriangleleft x^{R}\cdot-$ by assumption.
Also note that 
\[
-\left(x^{R}-x\right)y^{L}<\left(x-x^{L}\right)y^{L}
\]
 because $-\left(x^{R}-x\right)<x-x^{L}$ and $-\cdot y^{L}$ is order
preserving by assumption. To use Lemma \ref{lem: 15-1}, note that
\[
x^{L}\left(y-y^{L}\right)<x^{R}\left(y-y^{L}\right)
\]
 because $x^{L}\cdot-\vartriangleleft x^{R}\cdot-$ by assumption,
so 
\[
x^{L}y-x^{L}y^{L}<x^{R}y-x^{R}y^{L}
\]
 because of the ``distributive'' assumptions before the theorem.
Hence, 
\[
xy^{L}-x^{L}y^{L}+x^{R}y^{L}-xy^{L}<x^{R}y-x^{L}y
\]
 by Corollary \ref{cor: 18}, and finally
\[
\left(x-x^{L}\right)y^{L}-\left[-\left(x^{R}-x\right)y^{L}\right]<x^{R}y-x^{L}y.
\]
 Therefore, by Lemma \ref{lem: 15-1}, $A<D$.

To show $B<D$, note that $\left(x^{R}-x\right)y^{L}<\left(x^{R}-x\right)y^{R}$
since $-\cdot y^{L}\vartriangleleft-\cdot y^{R}$, so $-\left(x^{R}-x\right)y^{R}<-\left(x^{R}-x\right)y^{L}$
by Lemma \ref{lem: 14-1}. Hence, $x^{R}y-\left(x^{R}-x\right)y^{R}<x^{R}y-\left(x^{R}-x\right)y^{L}$.

To show $B<C$, note that 
\[
-\left(x^{R}-x\right)y^{R}<\left(x-x^{L}\right)y^{R}
\]
 because $-\left(x^{R}-x\right)<x-x^{L}$ and $-\cdot y^{R}:B$, and
note that 
\[
x^{L}y<x^{R}y
\]
 because $x^{L}\cdot-\vartriangleleft x^{R}\cdot-$. Similarly as
in the penultimate paragraph, 
\[
x^{R}y-x^{L}y<\left(x-x^{L}\right)y^{R}-\left[-\left(x^{R}-x\right)y^{R}\right],
\]
 so $B<C$ by Lemma \ref{lem: 15-1}.
\end{proof}
Recall that we need each $x\cdot-:B$ and we need each $x\mapsto x\cdot-$
to preserve $<$ and $\leq$. To do that, we prove the following:
\begin{thm}
Let $x,y:\mathsf{No}_{>0}$. Then $xy:\mathsf{No}_{>0}$. \label{thm: 18}\end{thm}
\begin{proof}
We use induction on $x$ and $y$. Note that $x,y$ are arbitrary,
so we will assume that the functions $x^{L}\cdot-$, $x^{R}\cdot-$
and $-\cdot y^{L},-\cdot y^{R}$ are positive. In the definition of
$xy$ above, we have $0<x^{L}y+\left(x-x^{L}\right)y^{L}$ by Lemma
\ref{lem: 14-1}(1), so $0<xy$ by Lemma \ref{lem: 15} and by transitivity%
\footnote{The relation $<$ on $\mathsf{No}$ is transitive. See \cite{key-24},
Corollary 11.6.17.%
} of $<$.
\end{proof}
Observe that $0\leq x$ and $0<y$ also imply $0<xy$. By induction,
$0\leq x^{L}y$ for all $y$ and $-\cdot y$ is a positive function,
so $0<x^{L}y+\left(x-x^{L}\right)y^{L}=\left(xy\right)^{L}$ by Theorem
\ref{thm: 16}(6). Hence, $0<xy$.
\begin{thm}[Distributive Law]
For all $x,y,z:\mathsf{No}_{>0}$, $x\left(y+z\right)=xy+xz$. \label{thm: 19}\end{thm}
\begin{proof}
We use induction on $x,y,z$. Let $P\left(x,y,z\right)$ be ``for
all $x,y,z:\mathsf{No}_{>0}$, $x\left(y+z\right)=xy+xz$.'' We assume
$P$ holds on triples $\left(x^{a_{1}},y^{a_{2}},z^{a_{3}}\right)$,
where each $a_{i}$ is $\emptyset$, $L$, or $R$, with the exception
that $a_{i}=\emptyset$ for all $i$, and we write $x^{\emptyset}$
for $x$. Under these assumptions, note that our definition of $xy$
above becomes 
\[
\left\{ x^{L}y+xy^{L}-x^{L}y^{L},x^{R}y-x^{R}y^{R}+xy^{R}|x^{L}y+xy^{R}-x^{L}y^{R},x^{R}y-x^{R}y^{L}+xy^{L}\right\} .
\]
 Upon expanding $x\left(y+z\right)$ and $xy+xz$ in a like manner,
it follows each left option of $x\left(y+z\right)$ is equal to a
left option of $xy+xz$, and vice versa, and the same goes for the
right options. Therefore, $x\left(y+z\right)=xy+xz$ by Lemma \ref{lem: 16}.
\end{proof}
Now that we have these results, we can prove that $x\cdot-:B$ for
each $x:\mathsf{No}_{>0}$ and that the function given by $x\mapsto x\cdot-$
is order preserving. At this point, the symbol $x\left(-y\right)$
is meaningless because $-y<0$, so we make the convention that $x\left(-y\right)\coloneqq-xy$. 
\begin{thm}
For all $x,x',y,y':\mathsf{No}_{>0}$,
\begin{enumerate}
\item $y<y'$ implies $xy<xy'$;
\item $y\leq y'$ implies $xy\leq xy'$;
\item $x<x'$ implies $xy<x'y$;
\item $x\leq x'$ implies $xy\leq x'y$.
\end{enumerate}
\end{thm}
\begin{proof}
(1) $y<y'$ implies $0<y'-y$, so $0<x\left(y'-y\right)$ by Theorem
\ref{thm: 18}, implying $0<xy'-xy$. Hence, $xy<xy'$. 

(2) - (4) are proved similarly.
\end{proof}
We are done with the multiplication on the positive surreal numbers.
Let us now show $x0=0$ for all $x:\mathsf{No}_{>0}$, which we will
use in the proof of the following theorem. By the distributive law
and the above convention that $x\left(-y\right)\coloneqq-xy$, we
have $x0=x\left(1-1\right)=x1-\left(x1\right)=0$. By a \textit{monoid}
we mean a set $M$ with a binary relation $\cdot$ and an element
$1$ such that for all $x,y,z:M$,
\begin{enumerate}
\item $x\cdot1=x$;
\item $x\cdot\left(y\cdot z\right)=\left(x\cdot y\right)\cdot z$, 
\end{enumerate}
and we say $M$ is \textit{commutative} if for all $x,y:M$, $x\cdot y=y\cdot x$.
We will omit the dot and write $xy$ for $x\cdot y$.
\begin{thm}
$\mathsf{No}_{>0}$ is a commutative monoid.\end{thm}
\begin{proof}
We use induction on $x$. Note that $1=\left\{ 0|\textrm{}\right\} $.
Let $x$ be given by $\left\{ x^{L}|x^{R}\right\} $, and assume $x^{L}1=x^{L}$
for all $L$ and $x^{R}1=x^{R}$ for all $R$. By the above definition
of $xy$, we have 
\[
x1=\left\{ x^{L}1+\left(x-x^{L}\right)0|x^{R}1-\left(x^{R}-x\right)0\right\} =\left\{ x^{L}|x^{R}\right\} =x.
\]
 Note that the left option $B$ and right $C$ disappear in the expansion
of $x1$ above. If $B$ and $C$ are seen as outputs of functions
of four variables, then there is no output when $y^{R}=1^{R}$ since
$1$ has no right option.

The associative law is proved by induction on $x,y,z$ and the commutative
law by induction on $x,y$, then by using Lemma \ref{lem: 16} as
in the proof of Theorem \ref{thm: 19}.
\end{proof}
A \textit{commutative ring with identity }$1$ is a set $R$ with
two binary relations $+$ and $\cdot$, a unary operation $-$, two
constants $0$ and $1$ such that $\left(R,+,-,0\right)$ is an abelian
group, $\left(R,\cdot,1\right)$ is a commutative monoid, and the
distributive law holds. 
\begin{thm}
$\mathsf{No}_{>0}$ is a commutative ring with identity $1$.
\end{thm}
We prove the following theorem, which will be of use in the next section:
\begin{thm}
Let $R$ be a commutative ring with identity $1$. Let $R-R$ be the
set $\left\{ a-b|a,b:R\right\} $. If $R-R$ has a multiplication
defined as $\left(a-b\right)*\left(a'-b'\right)\coloneqq aa'-ab'-ba'+bb'$,
then $\left(R-R,*,1\right)$ is a commutative monoid. \label{thm: 26}\end{thm}
\begin{proof}
Note that $\left(a-b\right)*\left(a'-b'\right)$ can be written as
$\left(aa'+bb'\right)-\left(ab'+ba'\right)$, so $R-R$ is closed
under $*$. We now prove the monoidal laws: 
\begin{enumerate}
\item Note also that $1=\left(1+1\right)-1$, so $\left(a-b\right)*1=\left(a-b\right)*\left(\left(1+1\right)-1\right)=a\left(1+1\right)-a1-b\left(1+1\right)+b1=a-b$.
\item For the associative law, we will omit $*$:
\[
\begin{array}{ccc}
\left(a-b\right)\left[\left(a'-b'\right)\left(a''-b''\right)\right] & = & \left(a-b\right)\left[a'a''-a'b''-b'a''+b'b''\right]\\
 & = & \left(a-b\right)\left[\left(a'a''+b'b''\right)-\left(a'b''+b'a''\right)\right]\\
 & = & a\left(a'a''+b'b''\right)-a\left(a'b''+b'a''\right)-b\left(a'a''+b'b''\right)+b\left(a'b''+b'a''\right)\\
 & = & a\left(a'a''\right)+a\left(b'b''\right)-a\left(a'b''\right)-a\left(b'a''\right)-b\left(a'a''\right)-b\left(b'b''\right)+b\left(a'b''\right)+b\left(b'a''\right)\\
 & = & \left(aa'\right)a''+\left(ab'\right)b''-\left(aa'\right)b''-\left(ab'\right)a''-\left(ba'\right)a''-\left(bb'\right)b'+\left(ba'\right)b''+\left(bb'\right)a''\\
 & = & \left[aa'-ab'-ba'+bb'\right]\left(a''-b''\right)\\
 & = & \left[\left(a-b\right)\left(a'-b'\right)\right]\left(a''-b''\right).
\end{array}
\]
 
\item The commutative law follows from the commutative law on $R$.
\end{enumerate}
\end{proof}
For a commutative ring $R$ with identity, a natural addition on $R-R$,
as defined in Theorem \ref{thm: 26}, is the addition on $R$ since
each element of $R-R$ is an element of $R$. With that addition,
we have 
\begin{thm}[Distributive Law]
For all $a-b,a'-b',a''-b'':R-R$, $\left(a-b\right)*\left[\left(a'-b'\right)+\left(a''-b''\right)\right]=\left(a-b\right)*\left(a'-b'\right)+\left(a-b\right)*\left(a''-b''\right)$,
where $*$ is as defined in Theorem \ref{thm: 26}. \label{thm: 27}\end{thm}
\begin{proof}
We will omit $*$. Note that $\left(a'-b'\right)+\left(a''-b''\right)=\left(a'+a''\right)-\left(b'+b''\right)$
since $R$ is a ring, so 
\[
\begin{array}{ccc}
\left(a-b\right)\left[\left(a'-b'\right)+\left(a''-b''\right)\right] & = & \left(a-b\right)\left[\left(a'+a''\right)-\left(b'+b''\right)\right]\\
 & = & a\left(a'+a''\right)-a\left(b'+b''\right)-b\left(a'+a''\right)+b\left(b'+b''\right)\\
 & = & aa'+aa''-ab'-ab''-ba'-ba''+bb'+bb''\\
 & = & \left(a-b\right)\left(a'-b'\right)+\left(a-b\right)\left(a''-b''\right).
\end{array}
\]

\end{proof}

\section{$\mathsf{No}$ as Differences of Positive Surreal Numbers}
\begin{lem}
For each $x:\mathsf{No}$, there is $n:\mathsf{Oz}$ such that $0<n$
and $x<n$. \label{lem: 21}\end{lem}
\begin{proof}
Let $x:\mathsf{No}$. By the Weak-Archimedean Condition%
\footnote{See Section 5.%
}, there is $m:\mathsf{No}$ such that $x\leq m+1$. If $m+2<0$ or
$m+2=0$, let $n=1$. Otherwise, let $n=m+2$.\end{proof}
\begin{thm}
Each surreal number is the difference of two positive surreal numbers.
\label{thm: 35}\end{thm}
\begin{proof}
Let $x:\mathsf{No}$. By Lemma \ref{lem: 21}, there is a positive
omnific integer $n$ such that $x<n$, so $0<n-x$ by Corollary \ref{cor: 18}.
Then $x=n-\left(n-x\right)$ by Lemma \ref{lem: 19}.
\end{proof}
In the proof of Theorem \ref{thm: 35}, the function $f:\mathsf{No}\rightarrow\mathsf{No}_{>0}-\mathsf{No}_{>0}$
given by $x\mapsto n-\left(n-x\right)$ is an equivalence. To see
this, the function $g:\mathsf{No}_{>0}-\mathsf{No}_{>0}\rightarrow\mathsf{No}$
given by $a-b\mapsto a-b$ is simply substraction in $\mathsf{No}$,
so $gf\left(x\right)=x$ for all $x:\mathsf{No}$, a consequence of
Lemma \ref{lem: 19}, and $fg\left(a-b\right)=a-b$ for all $a-b:\mathsf{No}_{>0}-\mathsf{No}_{>0}$.
Hence, we have
\begin{cor}
$\mathsf{No}\simeq\mathsf{No}_{>0}-\mathsf{No}_{>0}$. \label{cor: 36}
\end{cor}
Recall that $\mathsf{No}$ is a set by Theorem \ref{thm: 16-1}. We
will also show that $\mathsf{No}_{>0}-\mathsf{No}_{>0}$ is a set
through that equivalence, but we need the following:
\begin{thm}[\cite{key-24}, p. 118]
Let $A,B$ be types. If $f:A\rightarrow B$ is an equivalence, then,
for all $a,a':A$, the function $h:\left(a=a'\right)\rightarrow\left(f\left(a\right)=f\left(a'\right)\right)$
in Lemma \ref{lem: 15-2} is an equivalence. \end{thm}
\begin{cor}
Let $A,B$ be types. If $f:A\rightarrow B$ is an equivalence, then
$A$ is a set if and only if $B$ is a set. \label{cor: 38}\end{cor}
\begin{proof}
Let $B$ be a set. Let $x,y:A$ and $p,q:x=y$. Then $h\left(p\right)=h\left(q\right)$
in $f\left(x\right)=f\left(y\right)$, so $p=q$ by applying the inverse
of $h$ on $p,q$. The converse is proved similarly. \end{proof}
\begin{cor}
$\mathsf{No}_{>0}-\mathsf{No}_{>0}$ is a set. \end{cor}
\begin{proof}
By Theorem \ref{thm: 16-1} and Corollaries \ref{cor: 36} and \ref{cor: 38}.
\end{proof}

\section{Multiplication on $\mathsf{No}$}

We now are about to define a multiplication on $\mathsf{No}$ and
show that $\mathsf{No}$ is a commutative monoid. 
\begin{thm}
$\mathsf{No}_{>0}-\mathsf{No}_{>0}$ is a commutative monoid. \label{thm: 40}\end{thm}
\begin{proof}
By Theorem \ref{thm: 26}.\end{proof}
\begin{thm}
Let $A,B$ be sets. If $f:A\rightarrow B$ is an equivalence, then
$A$ is a commutative monoid if and only if $B$ is a commutative
monoid. \label{thm: 41}\end{thm}
\begin{proof}
Suppose $B$ is a commutative monoid. Let $a,a':A$ and let $g$ be
the inverse of $f$. Then we define $a\times a'\coloneqq g\left(f\left(a\right)f\left(a'\right)\right)$,
and let $1_{A}\coloneqq g\left(1_{B}\right)$. It follows that $f\left(a\times a'\right)=f\left(a\right)f\left(a'\right)$.
\begin{enumerate}
\item Note that $a\times1_{A}=g\left(f\left(a\right)1_{B}\right)=g\left(f\left(a\right)\right)=a$.
\item For $a,a,a'':A$, we have 
\[
\begin{array}{ccc}
a\times\left(a'\times a''\right) & = & g\left(f\left(a\right)f\left(a'\times a''\right)\right)\\
 & = & g\left(f\left(a\right)\left[f\left(a'\right)f\left(a''\right)\right]\right)\\
 & = & g\left(\left[f\left(a\right)f\left(a'\right)\right]f\left(a''\right)\right)\\
 & = & g\left(f\left(a\times a'\right)f\left(a''\right)\right)\\
 & = & \left(a\times a'\right)\times a''.
\end{array}
\]

\item The commutative law follows from the commutative law on $B$.
\end{enumerate}
The converse is proved similarly.\end{proof}
\begin{thm}
$\mathsf{No}$ has a multiplication $\times$ such that $\left(\mathsf{No},\times,1\right)$
is a commutative monoid. \label{thm: 43}\end{thm}
\begin{proof}
There is an equivalence $f:\mathsf{No}\rightarrow\mathsf{No}_{>0}-\mathsf{No}_{>0}$
by Corollary \ref{cor: 36}, and $\mathsf{No}_{>0}-\mathsf{No}_{>0}$
is a commutative monoid by Theorem \ref{thm: 40}. Therefore, a multiplication
can be defined on $\mathsf{No}$ as in the proof of Theorem \ref{thm: 41},
which turns $\mathsf{No}$ into a commutative monoid by the same theorem.
\end{proof}
We now will show that the distributive law holds on $\mathsf{No}$.
\begin{lem}
The distributive law holds on $\mathsf{No}_{>0}-\mathsf{No}_{>0}$.
\label{lem: 44}\end{lem}
\begin{proof}
By Theorem \ref{thm: 27}.
\end{proof}
Recall that the equivalence $f:\mathsf{No}\rightarrow\mathsf{No}_{>0}-\mathsf{No}_{>0}$
is defined as $x\mapsto n-\left(n-x\right)$, where $0<n$ and $x<n$. 
\begin{thm}
For all $x,y:\mathsf{No}$, $f\left(x\right)+f\left(y\right)=f\left(x+y\right)$.
\label{thm: 45}\end{thm}
\begin{proof}
We have 
\[
\begin{array}{ccc}
f\left(x\right)+f\left(y\right) & = & n-\left(n-x\right)+m-\left(m-y\right)\\
 & = & n+m-\left[\left(n-x\right)+\left(m-y\right)\right]\\
 & = & n+m-\left[\left(n+m\right)-\left(x+y\right)\right].
\end{array}
\]

Note that $0<n+m$ by Lemma \ref{lem: 14-1}(2) and $x+y<n+m$ by
Lemma \ref{lem: 14-1}(3). Thus, $f\left(x\right)+f\left(y\right)=f\left(x+y\right)$.
\end{proof}
Recall also that the inverse $\mathsf{No}\leftarrow\mathsf{No}_{>0}-\mathsf{No}_{>0}:g$
of $f$ is defined as $a-b\mapsto a-b$, where the right hand $a-b$
is substraction in $\mathsf{No}$, and addition on $\mathsf{No}_{>0}-\mathsf{No}_{>0}$
is addition on $\mathsf{No}$. We then have
\begin{thm}
For all $a-b,a'-b':\mathsf{No}_{>0}-\mathsf{No}_{>0}$, $g\left(\left(a-b\right)+\left(a'-b'\right)\right)=g\left(a-b\right)+g\left(a'-b'\right)$.
\label{thm: 46}\end{thm}
\begin{proof}
We have 
\[
\begin{array}{ccc}
g\left(\left(a-b\right)+\left(a'-b'\right)\right) & = & g\left(\left(a+a'\right)-\left(b+b'\right)\right)\\
 & = & \left(a+a'\right)-\left(b+b'\right)\\
 & = & \left(a-b\right)+\left(a'-b'\right)\\
 & = & g\left(a-b\right)+g\left(a'-b'\right).
\end{array}
\]
\end{proof}
\begin{thm}[Distributive Law]
For all $x,y,z:\mathsf{No}$, $x\times\left(y+z\right)=x\times y+x\times z$.\end{thm}
\begin{proof}
Recall that $\mathsf{No}$ has a multiplication $\times$ by Theorem
\ref{thm: 43}, and that multiplication is defined in the proof of
Theorem \ref{thm: 41}. We need only show $g\left(f\left(x\right)f\left(y+z\right)\right)=g\left(f\left(x\right)f\left(y\right)\right)+g\left(f\left(x\right)f\left(z\right)\right)$,
where $f:\mathsf{No}\rightarrow\mathsf{No}_{>0}-\mathsf{No}_{>0}$
is an equivalence and $g$ is its inverse, so 
\[
\begin{array}{cccc}
g\left(f\left(x\right)f\left(y+z\right)\right) & = & g\left(f\left(x\right)\left[f\left(y\right)+f\left(z\right)\right]\right) & \textrm{ by Theorem \ref{thm: 45}}\\
 & = & g\left(f\left(x\right)f\left(y\right)+f\left(x\right)f\left(z\right)\right) & \textrm{ by Lemma \ref{lem: 44} }\\
 & = & g\left(f\left(x\right)f\left(y\right)\right)+g\left(f\left(x\right)f\left(z\right)\right) & \textrm{ by Theorem \ref{thm: 46}}.
\end{array}
\]
\end{proof}
\begin{cor}
$\mathsf{No}$ is a commutative ring with identity $1$.
\end{cor}

\section{Apartness Relation on $\mathsf{No}$}

Given a type $A$, we call a relation $\neq:A\rightarrow A\rightarrow\mathcal{U}$
a \textit{difference relation }if, for all $a:A$, $\neg a\neq a$
and for all $a,b:A$, $a\neq b\rightarrow b\neq a$. A difference
relation on $A$ is an \textit{apartness} if, for all $a,b,c:A$,
$a\neq b\rightarrow a\neq c+c\neq b$ (cotransitivity). An apartness
is a positive notion for two elements being unequal; for more on apartnesses,
see \cite{key-20,key-21,key-28,key-8-1}. In $\mathsf{No}$, we write
$x\neq y$ to mean $\left\Vert x<y+y<x\right\Vert $. 
\begin{thm}
The relation $\neq$ on $\mathsf{No}$ is an apartness.\end{thm}
\begin{proof}
That $\neq$ is a difference relation follows from the asymmetry of
$<$ and from $A+B\simeq B+A$ for all types $A,B:\mathcal{U}$, and
cotransitivity of $\neq$ follows from cotransitivity of $<$ by Corollary
\ref{cor: 19}.
\end{proof}

\section{Problems}

\subsection{Axiomatizing the surreal numbers}

\textit{Can there be a possibly independent list of axioms for the
surreal numbers? }

If there can be such a list, it might be easier to prove theorems
about the surreal numbers because these axioms would capture what
the surreal numbers are. Here is an attempt, which is far from being
a complete answer. The surreal numbers are
\begin{enumerate}
\item a set $X$, as defined in Section 2;
\item $X$ has a binary relation $<$ that turns it into an ordered set%
\footnote{See Section 2.%
}. A binary relation $\leq$ can be defined either as $x\leq y$ if
$\neg y<x$ or as $x\leq y$ if for all $z:X$, $z<x$ implies $z<y$,
and $y<z$ implies $x<z$.
\item $X$ contains the omnific integers $\mathsf{Oz}$. 
\item $X$ satisfies the weak-Archimedean condition, as introduced in Section
5.
\end{enumerate}

\subsection{The surreal numbers as the field of fractions of the omnific integers:
a constructive take}

\textit{Is there a proof that each surreal number is the quotient
of two omnific integers, without the assumption that $<$ be trichotomous?}

In \cite{key-12-1}, each surreal number is said to be the quotient
of two omnific integers (Theorem 32), and the proof starts by picking
an arbitrary surreal number and considering its normal form, but to
show that each surreal number has a normal form (p. 32), the positive
surreal numbers are first considered. A similar move is used in \cite{key-14-1}
(See Theorems 8.3 and 5.6.). In either case, the normal form of a
surreal number $x$ is constructed when $x<0$, $x=0$, or $0<x$.
Note that the relation $<$ on the surreal numbers is not trichotomous%
\footnote{See Section 3.%
} in our development.

\end{document}